\documentclass[12pt,a4paper] {article}
\usepackage{latexsym,amsmath,enumerate,amssymb,amsbsy,amsthm}
\usepackage{enumerate,verbatim,tikz}

\newtheorem{theorem}{Theorem}[section]
\newtheorem{corollary}[theorem]{Corollary}
\newtheorem{lemma}[theorem]{Lemma}

\theoremstyle{definition}

\newtheorem{conjecture}{Conjecture}
\newtheorem{problem}{Problem}

\numberwithin{equation}{section}

\renewcommand\footnotemark{}

\usepackage{tikz}
\usetikzlibrary{shapes.geometric}
\usetikzlibrary{shapes.arrows}
\usepackage{array}

\begin{document}
	
	\title{Determinants of  some Special Matrices over Commutative Finite  Chain Rings\thanks{This research was supported by the Thailand Research Fund and  Silpakorn University under Research Grant RSA6280042.} }
	
	\author{Somphong~Jitman}

	\thanks{S. Jitman  is with the  Department of Mathematics, Faculty of Science,
		Silpakorn University, Nakhon Pathom 73000,  Thailand
		(email: sjitman@gmail.com).}


	\maketitle
	
	\begin{abstract}
Circulant matrices over finite fields and over commutative finite chain rings  have been of interest due to their nice algebraic structures and wide applications. In many cases, such matrices over rings have a closed connection with diagonal matrices over their extension rings. 
In this paper,  the determinants of  diagonal and circulant   matrices over  commutative finite chain rings $R$ with residue field $\mathbb{F}_q$    are studied.  	The number of  $n\times n$ diagonal  matrices over  ${R}$ of determinant  $a$  is determined for all elements $a$ in $ {R}$ and for all positive integers $n$.   Subsequently,  the enumeration  of  nonsingular $n\times n$ circulant  matrices over   ${R}$  of determinant   $a$  is given for all units $a$ in $ {R}$ and all positive integers $n$ such that $\gcd(n,q)=1$.    In some cases, the number of singular  $n\times n$ circulant  matrices over   ${R}$   with a fixed determinant is determined through the link between  the rings of  circulant matrices  and diagonal matrices.
As applications, a brief discussion on the determinant of diagonal and circulant matrices over  commutative finite principal ideal rings   is given. Finally, some open problems and conjectures are posted


	\end{abstract}

 \section{Introduction}
 Circulant   matrices  have been  introduced  in   \cite{c1846} and  extensively been  studied  due to their nice  algebraic structures, wide  applications,  various links with other objects.   The  book ``Circulant Matrices''  \cite{D1994} has summarized algebraic structures, properties and  applications of   such matrices. 
 Later, circulant  matrices have been shown to have applications in  many disciplines, e.g., signal processing, image processing, networked systems, communications,  and coding theory.    Especially, (nonsingular)  circulant  matrices  over finite fields and over commutative finite chain rings  are  applied in constructions of  various  families of linear codes  (see \cite{AOS2018},  \cite{D1994},    \cite{GSF2016},  \cite{GG2009}, \cite{GH1998},  \cite{KH2012}, \cite{KS2012},      \cite{SHSS2019},   \cite{SHSS2019-2}, \cite{SZQS2018} ,  and references therein).   Circulant  matrices have  shown to have a closed connection with diagonal matrices  (see, for example,   \cite{D1994} and \cite{KS2012}). Therefore, some properties of circulant matrices can be determined in terms of diagonal matrices.

 Determinants   of matrices are  known for their  useful properties and applications in  linear algebra, matrix theory, and  other  braces of Mathematics   and Engineering.    
 Since the singularity of matrices is useful in applications,  determinants and related properties of matrices  have been extensively studied.      The number of  $n\times n$ singular (resp., nonsingular) matrices over a finite field  has been  given  in \cite{M1984}.     The number of $n\times n$ matrices over   commutative finite chain rings (CFCRs)  of a fixed  determinant has been completely determined  in \cite{CJU2017}.   Diagonal and circulant matrices   are  two interesting subfamilies of the  ones in \cite{CJU2017}. 
 Therefore, it is  of natural interest to study  the determinants of such matrices   over  CFCRs.

 The paper is organized as follows.   Some  definitions and useful properties of CFCRs $R$ and matrices are recalled in Section~\ref{sec:pre}. In Section~\ref{sec:diag}, the  number $d_n(R,a)$ of $n\times n$ diagonal  matrices over   $R$  of   determinant $a$   is established for all elements  $a \in R$ and  for all positive integers  $n$.   In Section~\ref{sec:cir}, the  number $c_n(R,a)$ of $n\times n$  circulant   matrices over    $R$      is studied. In the case where $a$ is a unit in $R$, $c_n(R,a)$  is determined for all positive integers $n$  such that $\gcd(n,q)=1$.   For  a non-unit   $a \in R$, the number  $c_n(R,a)$  is given  for all positive integers  $n$  such that $n|(q-1)$.   Summary, remarks, conjectures, and open problems are given in  Section~\ref{sec:concl}.

 \section{Preliminaries} \label{sec:pre}
 
 Definitions, notations,  and some useful properties of   commutative finite chain rings and matrices are recalled.

 \subsection{Commutative Finite Chain Rings}
 A ring  $R$ with identity $1\neq 0$  is  called a
 \textit{commutative finite chain ring (CFCR)} if it is finite,  commutative, and  its  ideals are linearly ordered by
 inclusion. The properties of CFCRs required  in this paper are recalled in the discussion below. The reader is referred to    \cite{GF2002},  \cite{H2001}, and \cite{HLM2003}  for more details on  CFCRs.

 A   CFCR  is known to be  a principal ideal ring and  its 
 maximal ideal is unique.  Let $R$ be a CFCR  whose maximal ideal is generated by  $\gamma$.  The ideals in $R$  can be written  in the forms of 
 \[R\supsetneq \gamma R \supsetneq \gamma^2 R\supsetneq \dots \supsetneq \gamma ^{e-1} R \supsetneq \gamma ^{e} R=\{0\},\]
 for some positive integer $e$.  The {\em nilpotency  index} of a CFCR $R$ is defined to be the smallest
 positive integer $e$ such that    $\gamma^e=0$.    The  quotient ring $R/\gamma R$ forms  a finite  field  referred as  the {\em residue field} of $R$.  The   cardinality and the  characteristic 
 of  $R$  are  powers
 of the characteristic of  $R/\gamma R$. For a CFCR $R$, let  $U(R)$ denote  the set of units in $R$ and let $Z(R)$ denote the set of zero-divisors in $R$.  The properties of a CFCR in the next lemma are well known.

 \begin{lemma}[{\cite{H2001} and \cite{HLM2003}}]  \label{lem:propCR}
 	Let $R$ be a CFCR of nilpotency index $e$ and let $\gamma$ be  a generator of the maximal ideal of $R$.   Let $ V \subseteq R$  be  a set of representatives for
 	the equivalence classes of $R$ under congruence modulo $\gamma$.
 	Assume that the residue field $R/\langle \gamma\rangle\cong \mathbb{F}_{q}$ for some prime power $q$. 
 	Then the following statements hold.
 	\begin{enumerate}[$1)$]
 		\item For each $r\in R$,  there exist  unique  $a_0, a_1 ,\dots
 		a_{e-1}\in
 		V$ such that $$r=a_0+a_1\gamma+\dots+a_{e-1}\gamma^{e-1}. $$
 		\item  $|V| = q$.
 		\item $| \gamma^j R|=
 		q^{e-j}$  for all $0\leq j\leq e$.
 		\item $U(R)= \{a+\gamma b\mid a\in  V\setminus\{0\} \text{ and } b\in R\}$.
 		\item $|U(R)|=(q-1)q^{e-1}$.
 		\item For each  $0\leq i\leq e$, 
 		$ R/\gamma^iR$ is a CFCR of nilpotency index $i$ and residue field $\mathbb{F}_q$.
 	\end{enumerate}
 \end{lemma}
 
 From Lemma~\ref{lem:propCR},  it can be deduced that   $Z(R)=\gamma R\setminus \{0\}$,   $|Z(R)|= q^{e-1}-1$, and $R=\{0\}\cup  Z(R)\cup U(R)$ is a disjoint union.

 \subsection{Diagonal and Circulant Matrices}

 Given a positive integer $n$ and a commutative ring $\mathfrak{R}$,  an $n\times n$ matrix  $A$ over $\mathfrak{R}$ is called a {\em diagonal matrix} if    $a_{ij}=0$ for all $i\ne j$. Denote by ${\rm diag}(a_{11}, a_{22},\dots, a_{nn})$ the diagonal matrix $A$.    It is well known that \[\det({\rm diag} (a_1,a_2,\dots,a_n) )=\prod_{i=1}^{n}a_i.\]
 Let $D_n(\mathfrak{R})=\{ {\rm diag} (a_1,a_2,\dots,a_n) \mid a_i\in \mathfrak{R}\}$ denote the set of $n\times n$ diagonal matrices over $\mathfrak{R}$.   For each $n\in \mathbb{N} $ and  $a\in \mathfrak{R}$,  let $D_n(\mathfrak{R},a)=\{ A\in D_n(\mathfrak{R})\mid \det(A)=a\}$ and $d_n(\mathfrak{R},a)=|D_n(\mathfrak{R},a)|$.

 An $n\times n$ matrix $A$ over $\mathfrak{R}$ is called a {\em circulant matrix} if $A$ is of  the form
 \[A=\begin{bmatrix}
 a_1&a_2&a_3&\dots& a_n\\
 a_{n}&a_1&a_2&\dots& a_{n-1}\\
 a_{n-1}&a_{n}&a_1&\dots& a_{n-2}\\
 \vdots & \vdots &\vdots & \ddots &\vdots \\
 a_{2}&a_{3}&a_4&\dots& a_{1}\\
 \end{bmatrix},
 \]
 for some $a_1,a_2,\dots,a_n$ in  $\mathfrak{R}$, and denoted its by ${\rm cir}( a_1,a_2,\dots,a_n)$. 
 If there exists an  extension ring of  $\mathfrak{R}$  containing a primitive $n$th root of unity, say  $\omega$, the eigenvalues and determinant of a circulant matrix  $A={\rm cir}( a_1,a_2,\dots,a_n) $ over $ \mathfrak{R}$ can be given as follows.  From \cite{D1994} and \cite{KS2012},   the eigenvalues of $A$ are of the form 
 \begin{align} \label{eig}
 w_j=\sum_{i=1}^n a_i\omega^{(i-1)j}\end{align} for all $0\leq j\leq n-1$.
 It follows that 
 \[ \det(A) =\prod_{j=0}^n \omega_j= \prod_{j=0} ^{n-1}  \left(\sum_{i=1}^n a_i\omega^{(i-1)j}\right).\]
 Let  $C_n(\mathfrak{R})=\{ {\rm cir}( a_1,a_2,\dots,a_n)\mid  a_i\in \mathfrak{R}\}$ denote the set of $n\times n$ circulant matrices over the ring $\mathfrak{R}$.  
 Let   $C_n(\mathfrak{R},a) =\{ A\in C_n( \mathfrak{R}) \mid \det(A)=a\}$  be the set of $n\times n$ circulant matrices over $\mathfrak{R}$ whose   determinant  is  $a$ and let $c_n(\mathfrak{R},a)=|C_n(\mathfrak{R},a)|$.

 In  this paper, we focus on  the numbers $c_n(R,a)$ and  $d_n(R,a)$ in the case where  $R$  is  CFCRs and $a\in R$ which are established   in Section~\ref{sec:diag} and Section~\ref{sec:cir}, respectively.

 \section{Determinants of Diagonal Matrices  over CFCRs}
 \label{sec:diag}
 
 Determinants of diagonal matrices over a CFCR $R$  are focused on and the number $d_n({R},a)$ is completely determined for all positive integers $n$ and for all $a\in R$. 
 
 For  a CFCR  $R$   of nilpotency index $e$ and residue field $\mathbb{F}_q$,    let $\gamma$  be a generator   its maximal ideal. For each $a\in R$,  it is easily seen that $a =\gamma^sb$   for some $0\leq s\leq e$ and unit  $b\in U(R)$ by Lemma~\ref{lem:propCR}.  Precisely, $a$ is a unit if $s=0$, 
 $a =\gamma^sb$ is  a zero-divisor if $1\leq s\leq e-1$, and  
 $a=0$ if $s=e$.

 For each $n\in \mathbb{N}$ and $a\in R$, the  number  $d_n(R,a)$   are   determined  in this section. The over view results are   summarized in Figure~\ref{D1} and the details
 are given right after.

 \begin{figure}[!hbt]
 	\centering
 	{\scriptsize
 		\begin{tikzpicture} [
 		block/.style    = { rectangle, draw=black, 
 			fill=black!5, text width=13em, text centered,
 			rounded corners, minimum height=2em },
 		line/.style     = { draw, thick, ->, shorten >=2pt },
 		]
 		\matrix [column sep=5mm, row sep=8  mm] {
 			
 			& \node [block] (oria) {$a \in R$ };            & \\ %
 			
 			& \node [block] (a) {$a=\gamma^s b$, $0\leq s \leq e$ and $b \in U(R)$ }; \\
 			
 			\node [block] (bri) {$a=\gamma^s b$, $0\leq s <e$   }; &  &\node [block] (a0) {$a =0$ };  \\ 
 			
 			\node [block] (aur1) {$d_n(R,a)=d_n(R,\gamma^s)$ };  && \\
 			
 			\node [block] (au1) {$a=1$ };  	   &\node [block] (rs) {$a=\gamma^s$ };  &	\node [block] (a01) {$d_n(R,0)$ };  \\
 			
 			\node [block] (au11) {$d_n(R,1) $ };  	   &\node [block] (rs1) {$ d_n(R,\gamma^s)$ };  &	  \\
 			
 			&\node [block] (end) {$ d_n(R,a)$ };  &	\\ 
 		};
 		\begin{scope} [every path/.style=line]
 		\path (oria)        --   node{\quad\quad \quad\quad \quad \quad ~ Lemma~\ref{lem:propCR}}   (a); 
 		
 		\path (a)        --   node{$0\leq s <e$\quad\quad\quad \quad \quad  \quad \quad }   (bri); 
 		\path (a)        --  node{\quad \quad \quad \quad $s=e$}  (a0);

 		\path (bri)        --    node{\quad\quad\quad  \quad\quad\quad ~ ~ Theorem ~\ref{ar=rD}}(aur1);  
 		
 		\path (aur1)        --    node{\quad\quad\quad $e=0$}(au1);  
 		\path (aur1)        --     node{\quad\quad\quad \quad\quad\quad \quad\quad $1 \leq s<e$}(rs);  
 		
 		\path (a0)        --    node{\quad\quad\quad \quad\quad\quad  ~ ~  Theorem~\ref{thma0D}}(a01);  
 		
 		\path (au1)        --  node{\quad\quad\quad \quad\quad\quad ~ ~  Corollary~\ref{a=1D}}  (au11);  
 		\path (rs)        --   node{\quad\quad\quad \quad\quad\quad   ~ ~ Theorem~\ref{a=rD}} (rs1);

 		\path (au11)        --    (end);  
 		\path (rs1)        --    (end);  
 		\path (a01)        --    (end);  
 		\end{scope}
 		%
 		\end{tikzpicture}
 	}
 	\caption{The number $d_n(R,a)$ over a CFCR  $R$} \label{D1}
 \end{figure}
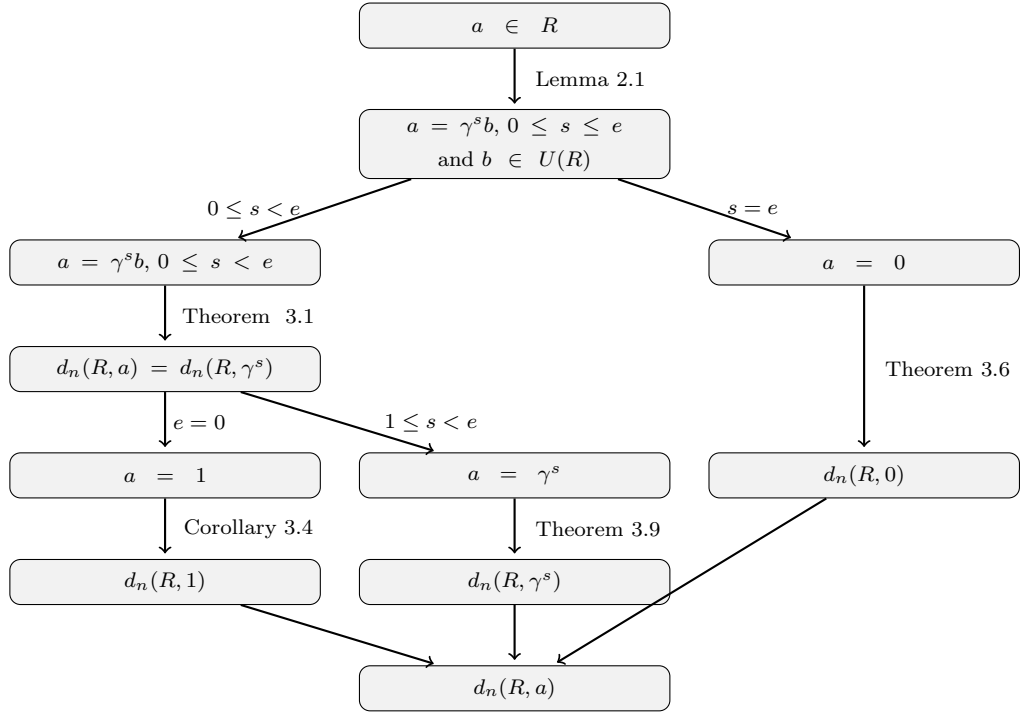

 First, we begin with a reduced  formula for  the number  $d_n(R,a)$. 
 
 \begin{theorem} \label{ar=rD} Let $R$ be a CFCR of nilpotency index $e$ and let $n$ be a positive integer. If the maximal ideal of $R$ is generated by $\gamma$ and $0\leq s\leq e$, then the following statements hold.
 	\begin{enumerate}[$1)$]
 		\item  $d_n(R,a)>0$ for all elements  $a\in R$.
 		\item $d_n(R,\gamma^s)=d_n(R,b\gamma^s)$ for all units $b$ in $U(R)$.  
 	\end{enumerate}
 \end{theorem}
 \begin{proof}    
 	Since  $\det( {\rm diag}(a,1,1,\dots,1))=a$ for all elements  $a\in R$,   it follows that  ${\rm diag}(a,1,1,\dots,1)\in D_n(R,a) $ which implies that  $d_n(R,a) >0$  for all $a\in R$.  This proves~$1)$.

 	To prove $2)$, let $b$ be a unit in $U(R)$ and let $0\leq s \leq e$ be an integer. If $s=e$, then $ \gamma^s=0=\gamma^sb$.  Clearly,  $d_n(R,\gamma^s)=d_n(R,0) =d_n(R,b\gamma^s)$.

 	For each $0\leq s<e$, let $\alpha: D_n(R,\gamma^s)\to D_n(R,b\gamma^s)$  be a map defined by 
 	\[\alpha(A)=  {\rm diag}(b,1,1,\dots,1) A\]
 	for all $A\in D_n(R,\gamma^s)$.   Since  $ {\rm diag}(b,1,1,\dots,1)$ is nonsingular and  $\det(A)= \gamma^s$ if  and only if $\det(   {\rm diag}(b,1,1,\dots,1)A)=b\det(A) =b\gamma^s$  for all  $A\in D_n(R,\gamma^s) $,        $\alpha$ is a  well-defined   bijective map.  Therefore, 
 	\[d_n(R,\gamma^s)=|D_n(R,\gamma^s)|=|D_n(R,b\gamma^s)|=d_n(R,b\gamma^s).\]
 	The second statement is proved. 
 \end{proof}

 By setting $s=0$,    the next corollary follows. 
 \begin{corollary} \label{a1D} Let $R$ be a CFCR and let $n$ be a positive integer.  Then  $d_n(R,a)=d_n(R,1)$ for all units $a\in U(R)$.
 \end{corollary}

 \subsection{Determinants of Nonsingular Diagonal Matrices over CFCRs}
 
 Nonsingular $n\times n$  diagonal matrices over CFCRs $R$ are focused on and the number $d_n(R,a)$ is determined for all positive integers $n$ and for all units $a\in U(R)$. 
 
 Let $NSD_n(R) =\{ A\in D_n(R) \mid \det(A) \in U(R)\}$ denote the set of  nonsingular $n\times n$  diagonal matrices over CFCRs $R$.  Clearly, \[NSD_n(R) =\bigcup_{a\in U(R)} D_n(R,a),\] where the union is disjoint. 
 
 First, we determined the number $|NSD_n(R)|$ of  nonsingular $n\times n$  diagonal matrices over CFCRs $R$. 
 \begin{lemma} \label{NSD}
 	Let $R$ be a CFCR of nilpotency index $e$ and residue field $\mathbb{F}_q$ and let $n$ be a positive integer. Then  \[|NSD_n(R)|= (q-1) ^nq^{(e-1)n}.\]
 \end{lemma}
 \begin{proof}
 	Since  $\det({\rm diag} (a_1,a_2,\dots,a_n) )\in U(R)$ if and only if $a_i\in U(R)$ for all $1\leq i\leq n$,  we have $|NSD_n(R)|= |U(R)|^n= (q-1) ^nq^{(e-1)n}$ by Lemma~\ref{lem:propCR}.
 \end{proof}

 By Corollary~\ref{a1D},     $d_n(R,a) =d_n(R,1)$ for a unit  $a\in R$, and hence, $d_n(R,1)=|NSD_n(R)|/|U(R)|$.  By Lemma~\ref{lem:propCR} and Lemma~\ref{NSD}, the result follows.

 \begin{corollary}\label{a=1D} Let $R$ be a CFCR of nilpotency index $e$ and residue field $\mathbb{F}_q$ and let $n$ be a positive integer. Then  \[d_n(R,a)=d_n(R,1)=  (q-1) ^{(n-1)}q^{(e-1)(n-1)} \]  
 	for all $a\in U(R)$.
 \end{corollary}
 By setting $e=1$ in Lemma Lemma~\ref{NSD} and Corollary~\ref{a=1D}, it follows that 
 \begin{align} \label{dn1F}
 |NSD_n(\mathbb{F}_q)| = (q-1)^{n} \text{ ~  and ~ } d_n(\mathbb{F}_q,1)=(q-1)^{n-1}.
 \end{align}
 
 \subsection{Determinants of Singular Diagonal Matrices over CFCRs} 
 Determinants of singular $n\times n$ diagonal matrices over CFCRs $R$ are studied. Precisely, $d_n(R,0)$ and $d_n(R,a)$ are determined for all zero-divisors $a\in Z(R)$.  
 
 \subsubsection{Singular Diagonal Matrices over CFCRs with Zero Determinant}

 The recursive relation  in the  following lemma  is key in  determining    $d_n(R,0)$ in Theorem~\ref{thma0D}.

 \begin{lemma}\label{lema0}
 	Let $R$ be a CFCR of nilpotency index $e$ and residue field $\mathbb{F}_q$ and let $\gamma$  be a generator of the maximal ideal of $R$.  Then \begin{align*} d_n(R,0)&= (q-1)q^{e-1} d_{n-1}(R,0)  +q^{n-1}d_n(R/\gamma^{e-1}R,0+\gamma^{e-1}R)\end{align*} for all integers $n\geq 2$.
 \end{lemma}
 
 \begin{proof}
 	Let   \[D^\prime_n(R,0)=\{ {\rm diag}(a_1,a_2,\dots, a_n)\in D_n(R,0)\mid   a_{1}\in U(R) \}\] and  \[D^{\prime\prime}_n(R,0)=\{{\rm diag}(a_1,a_2,\dots, a_n)\in D_n(R,0) \mid  a_{1}\notin U(R)  \}.\]

 	Let $\nu: \rho(D^\prime_n(R,0) )\to D_{n-1}(R,0)$ be the map defined by 
 	\[{\rm diag}(a_1,a_2,\dots, a_n)\in D_n(R,0) \mapsto {\rm diag}(a_2,\dots, a_n)\in D_n(R,0). \]
 	It is not difficult to see that  $\nu$ is a   $(q-1)q^{e-1} $-to-one surjective  map. 
 	For each $A\in D^\prime_n(R,0)$,   it can be seen that  $\det(A)=0$ if and only if $\det(\rho(A))=0$.  Consequently,    \[|D^\prime_n(R,0)|= (q-1)q^{e-1}d_{n-1}(R,0).\]

 	For each ${\rm diag}(a_1,a_2,\dots, a_n) \in D^{\prime\prime}_n(R,0)$,   we have   $a_{1}\in \gamma R$  which implies that     $a_{1}=\gamma b_1$ for some  $b_1\in \sum\limits_{j=0}^{e-2} \gamma^jV$, where $ V$  is given  in    Lemma~\ref{lem:propCR}.   
 	Let   $\psi: D^{\prime\prime}_n(R,0) \to D_n(R)$ be an injective map defined  by 
 	\[{\rm diag}(a_1,a_2,\dots, a_n) \mapsto  {\rm diag}(b_1,a_2,\dots, a_n) .\]
 	Let  $\beta:D_n(R)\to D_n(R/ \gamma^{e-1}R)$ be a surjective ring homomorphism given~by 
 	\[\beta(B)=\overline{  B},\]
 	where  $\overline{[b_{ij}]}:=[b_{ij} +\gamma^{e-1}R]$ for all $[b_{ij}]\in D_n(R)$.  
 	For each $A\in D^{\prime\prime}_n(R,0) $,  $\det(A)=\gamma\det( \psi(A))$ which implies that    $\det(A)=0$  if and only if  $\det(\psi(A))\in \gamma^{e-1} R$. Consequently,   we have 
 	$\det(\beta(\psi(A)))=\det(\psi(A))  +\gamma^{e-1} R = 0 +\gamma^{e-1} R$.  Hence,  $\beta\circ \psi$ is  surjective such that  $\beta(\psi (D^{\prime\prime}_n(R,0) )) = D_n(R/\gamma^{e-1} R,0+\gamma^{e-1} R)$.
 	For each  $D\in D_n(R/\gamma^{e-1} R,0+\gamma^{e-1} R) $,  there are exactly  $  q^{n-1}$ matrices  in  $\psi (D^{\prime\prime}_n(R,0) ) $ whose images under $\beta$ are $D$. Since $\psi$ is an injective map,   we have  \[| D^{\prime\prime}_n(R,0)|= q^{n-1} d_n(R/\gamma^{e-1}R,0+\gamma^{e-1}R).\]
 	
 	Since  $D_n(R,0)=D^\prime_n(R,0)\cup D^{\prime\prime}_n(R,0)$ is a disjoint union, the desired result follows.
 \end{proof}

 %
 %
 %
 %

 Using the above relation, the number $d_n(R,0)$    can be  given in the next theorem.

 \begin{theorem} \label{thma0D} Let $R$ be a CFCR of nilpotency index $e$ and residue field $\mathbb{F}_q$ and let $n$ be a positive integer.  
 	Then \begin{align}\label{eqa0} d_n(R,0)&=q^{ne} -(q-1)^n q^{(e-1)n} \sum_{i=0}^{e-1} \binom{n+i-1}{n-1} q^{-i}.
 	\end{align}
 \end{theorem}
 \begin{proof} We prove  \eqref{eqa0} by   induction on the nilpotency index  $e$ and  the size  $n$.  For the case $e=1$,  we have  $R=\mathbb{F}_q$ and \eqref{eqa0}  holds by \eqref{dn1F},  i.e.,   \[d_n(R,0)=d_n(\mathbb{F}_q,0)=q^n-|NSD_n(\mathbb{F}_q)| = q^{n} - \left(1-q\right)^n .\] If $n=1$, then  $c_n(R,0)=1$   which coincides with \eqref{eqa0}.

 	Assume that \eqref{eqa0} holds true  for all integers $k\in\{1,2,\dots, n-1\}$ and  integers  $f\in\{1,2,\dots,e-1\}$.  Then  
 	\begin{align*}
 	d_k(R,0)= &(q-1)q^{f-1}d_{k-1}(R,0)
 	+q^{k-1}d_k(R/\gamma^{f-1}R,0+\gamma^{f-1}R)  \\ & \quad \text{ by  Lemma~\ref{lema0},}\\
 	= & (q-1)q^{f-1}  \left( q^{(k-1)f} -(q-1)^{k-1} q^{(f-1)(k-1)} \sum_{i=0}^{f-1} \binom{k+i-2}{k-2} q^{-i}\right)  \\
 	&+ q^{k-1} \left(q^{k(f-1)} -(q-1)^k q^{(f-2)k} \sum_{i=0}^{f-2} \binom{k+i-1}{k-1} q^{-i}\right)\\
 	& \quad \text{ by the induction  hypothesis,}\\ 
 	= &  (q-1)q^{kf-1} -  (q-1)^{k} q^{(f-1)k} \sum_{i=0}^{f-1} \binom{k+i-2}{k-2} q^{-i}  \\
 	&+ q^{kf-1} - (q-1)^k q^{(f-1)k-1} \sum_{i=0}^{f-2} \binom{k+i-1}{k-1} q^{-i}\\
 	= &  (q-1)q^{kf-1} -  (q-1)^{k} q^{(f-1)k}   \sum_{i=0}^{f-1} \binom{k+i-2}{k-2} q^{-i}   \\
 	&+ q^{kf-1} - (q-1)^k q^{(f-1)k}  \sum_{i=0}^{f-1} \binom{k+i-2}{k-1} q^{-i} \\
 	= &  q^{kf} -  (q-1)^{k} q^{(f-1)k}   \sum_{i=0}^{f-1} \left(  \binom{k+i-2}{k-2}    + \binom{k+i-2}{k-1}\right) q^{-i-1}\\
 	= &  q^{kf} -      (q-1)^k q^{(f-1)k}
 	\sum_{i=0}^{f-1} \binom{k+i-1}{k-1} q^{-i}.
 	\end{align*}
 	The result   follows immediately.
 \end{proof}

 \subsubsection{Singular Diagonal Matrices over CFCRs with Non-Zero Determinant}
 Here, we focus on singular  diagonal matrices whose determinant is non-zero. Precisely, the number  $d_n(R,a) $  is determined for all  zero-divisors $a\in Z(R)$.  
 Note that for  each zero-divisor $a\in Z(R)$, it can be written as  $a=\gamma^sb$ for some $1\leq s<e$ and $b\in U(R)$. By Theorem~\ref{ar=rD}, it is enough to determine only  $d_n(R,\gamma^s) $  for all integers  $1\leq s<e$. 
 
 The following results   are     key tools in determining  $d_n(R,\gamma^s) $  in Theorem~\ref{a=rD}.

 \begin{lemma}\label{lems} Let $R$ be a CFCR of nilpotency index $e\geq 3$ and residue field $\mathbb{F}_q$ and let  $\gamma$ be  a generator of the maximal ideal of $R$.  Then \[d_n(R,\gamma^s)=q^{(n-1)}d_n(R/\gamma^{e-1}R,\gamma^s+ \gamma^{e-1}R)\] for all $1\leq s<e-1$ and for all integers $n\geq 2$.
 \end{lemma}
 \begin{proof}
 	Let $n\geq 2 $ and $1\leq s<e-1$ be  integers and let   $\beta:D_n(R)\to D_n(R/ \gamma^{e-1}R)$ be a ring homomorphism defined in the proof of  Lemma~\ref{lema0}~by 
 	\[\beta(A)=\overline{A},\]
 	where  $\overline{[a_{ij}]}:=[a_{ij} +\gamma^{e-1}R]$ for all $[a_{ij}]\in D_n(R)$.
 	For each $A\in D_n(R)$,  we have   $\det(\beta(A))=\gamma^s+ \gamma^{e-1}R$ if and only if $\det (A)= \gamma^s+ \gamma^{e-1}b$  for some  $b\in V$, where $V$ is given in Lemma~\ref{lem:propCR}.
 	Since $1\leq e-s-1< e-1$,   the element  $1+\gamma^{e-s-1}b$ is a unit in $U(R)$. Consequently, 
 	\begin{align*}
 	|\{A\in D_n(R)\mid  &\det(A) = \gamma^s+ \gamma^{e-1}b \text{ for some } b\in V\}|\\
 	&= |\{A\in D_n(R)\mid \det(A) = \gamma^s(1+ \gamma^{e-s-1}b)\text{ for some } b\in V\}|\\
 	&= |\{A\in D_n(R)\mid \det(A) = \gamma^s \}|\\
 	&=d_n(R, \gamma^s).
 	\end{align*}	  
 	Equivalently, 
 	\begin{align}\label{eq11}
 	|\{A\in D_n(R)\mid \det(\beta(A)) = \gamma^s+ \gamma^{e-1}R \}|
 	&= |V|c_n(R, \gamma^s)= qd_n(R, \gamma^s).
 	\end{align}	  
 	It is not difficult to see that $\ker (\beta)= D_n(\gamma^{e-1}R) $ and   $|\ker (\beta)| =q^{n}$. Hence,
 	\begin{align}\label{eq22}
 	|\{A\in D_n(R)\mid &\det(\beta(A)) = \gamma^s+ \gamma^{e-1}R  \}|\notag \\
 	&= q^{n}|\{B\in D_n(R/\gamma^{e-1}R)\mid \det(B) = \gamma^s +\gamma^{e-1}R\}|\notag\\
 	&= q^{n}d_n(R/\gamma^{e-1}R, \gamma^s+\gamma^{e-1}R).
 	\end{align}	  
 	From \eqref{eq11} and \eqref{eq22},  it can be concluded that  
 	\[ qd_n(R, \gamma^s) = q^{n}d_n(R/\gamma^{e-1}R, \gamma^s+\gamma^{e-1}R).\] 
 	As desired, we have  \[d_n(R,\gamma^s)=q^{(n-1)}d_n(R/\gamma^{e-1}R,\gamma^s+ \gamma^{e-1}R).\]   The proof is completed.
 \end{proof}
 
 Applying Lemma~\ref{lems} recursively,  the next  corollary follows.
 \begin{corollary}\label{corlast} Let $R$ be a CFCR of nilpotency index $e+f$ and residue field $\mathbb{F}_q$, where $2\leq e$ and $1\leq f$ are integers.
 	Let $\gamma$ be a generator of the maximal ideal of $R$.  Then \[d_n(R,\gamma^s)=q^{f(n-1)}d_n(R/\gamma^eR,\gamma^s+ \gamma^eR)\] for all $1\leq s<e$ and for all positive   integers $n$.
 \end{corollary}

 The number  $d_n(R,\gamma^s) $   is determined in the next theorem.
 
 \begin{theorem}\label{a=rD} Let $R$ be a CFCR of nilpotency index $e$ and residue field $\mathbb{F}_q$ and let  $\gamma$ be a generator of  the maximal ideal of $R$.  Then \[c_n(R,\gamma^s)=q^{(e-1)(n-1)}   (q-1)^{n-1}  \binom{n+s-1}{n-1}  \]
 	for all  integers $1\leq s<e$ and for all positive integers $n$.
 \end{theorem}

 \begin{proof}
 	Let   $1\leq s<e$ be an integer and let $n$ be a positive integer.  Let  $\mu:D_n(R/ \gamma^{s+1}R)\to D_n(R/ \gamma^{s}R)$ be a ring homomorphism defined by 
 	\[\mu(A)=\overline{A},\]
 	where  $\overline{[a_{ij} +\gamma^{s+1}R]}:=[a_{ij} +\gamma^{s}R]$ for all $[a_{ij}+\gamma^{s+1}R]\in D_n(R/ \gamma^{s+1}R)$.  For each $A\in D_n(R/ \gamma^{s+1}R)$, we then have  $\det(\mu(A))=0+\gamma^{s}R $ if and only if $\det(A)=\gamma^{s} b+\gamma^{s+1}R$ for some $b\in V$, where $V$ is given  in Lemma~\ref{lem:propCR}.   Since $|\ker(\mu)|=q^{n}$,  it follows that 
 	\begin{align*}
 	q^{n} d_n&(R/\gamma^{s}R,0+\gamma^{s}R)\\
 	&= |\ker(\mu)|d_n(R/\gamma^{s}R,0+\gamma^{s}R)\\
 	&=\sum_{b\in V} d_n(R/\gamma^{s+1}R,\gamma^{s}b+\gamma^{s+1}R)\\
 	&=d_n(R/\gamma^{s+1}R,0+\gamma^{s+1}R) +\sum_{b\in V\setminus \{0\}} d_n(R/\gamma^{s+1}R,\gamma^{s}b+\gamma^{s+1}R)  \\
 	&=d_n(R/\gamma^{s+1}R,0+\gamma^{s+1}R) +(q-1)	d_n(R/\gamma^{s+1}R,\gamma^{s}+\gamma^{s+1}R)
 	\end{align*}
 	by  Theorem~\ref{ar=rD}.   Consequently, 
 	\begin{align} \label{eq1}
 	d_n(&R/\gamma^{s+1}R,\gamma^{s}+\gamma^{s+1}R)\notag  \\
 	&= \frac{1}{q-1}\left(q^{n} d_n(R/\gamma^{s}R,0+\gamma^{s}R)-d_n(R/\gamma^{s+1}R,0+\gamma^{s+1}R)\right).
 	\end{align}
 	By Corollary~\ref{corlast},  it can be deduced that 
 	\begin{align} \label{eq2}
 	d_n(R,\gamma^{s}) &= 	d_n(R/\gamma^{e+1+(s-e-1)}R,\gamma^{s}+\gamma^{e+1+(s-e-1)}R) \notag \\
 	&=q^{(e-s-1)(n-1)}  d_n(R/\gamma^{s+1}R,\gamma^{s}+\gamma^{s+1}R).
 	\end{align}
 	From (\ref{eq1}) and (\ref{eq2}),   we have 
 	\begin{align*}  
 	d_n(R,\gamma^{s})  
 	&=\frac{q^{(e-s-1)(n-1)} }{q-1} \Bigg(q^{n} d_n(R/\gamma^{s}R,0+\gamma^{s}R)  -d_n(R/\gamma^{s+1}R,0+\gamma^{s+1}R)\Bigg).
 	\end{align*}
 	Using Theorem~\ref{thma0D},  it follows that
 	\begin{align*}  
 	d_n(R,\gamma^{s})  
 	=& \frac{q^{(e-s-1)(n-1)} }{q-1} \Bigg(q^{n}  
 	\left( q^{ns} -(q-1)^n q^{(s-1)n} \sum_{i=0}^{s-1} \binom{n+i-1}{n-1} q^{-i}  \right) \\
 	& - \left( q^{n(s+1)} -(q-1)^n q^{sn} \sum_{i=0}^{s} \binom{n+i-1}{n-1} q^{-i}\right)\Bigg)\\
 	=& \frac{q^{(e-s-1)(n-1)} }{q-1} \Bigg( q^{n(s+1)} -(q-1)^n q^{sn} \sum_{i=0}^{s-1} \binom{n+i-1}{n-1} q^{-i}   \\
 	& - q^{n(s+1)} +(q-1)^n q^{sn} \sum_{i=0}^{s} \binom{n+i-1}{n-1} q^{-i}\Bigg)
 	\end{align*}
 	\begin{align*}\quad\quad
 	= &\frac{q^{(e-s-1)(n-1)} }{q-1} \Bigg( (q-1)^n q^{sn}\Bigg( \sum_{i=0}^{s} \binom{n+i-1}{n-1} q^{-i}  \\
 	&-\sum_{i=0}^{s-1} \binom{n+i-1}{n-1} q^{-i}  \Bigg)\Bigg)\\
 	=& \frac{q^{(e-s-1)(n-1)} }{q-1}  (q-1)^n q^{sn}  \binom{n+s-1}{n-1} q^{-s}   \\
 	=& q^{(e-1)(n-1)}   (q-1)^{n-1}  \binom{n+s-1}{n-1}    
 	\end{align*}
 	as desired. 
 \end{proof}

 \section{Circulant Matrices over CFCRs}
 \label{sec:cir}

 Determinant of    circulant matrices  over a CFCR are focused on an the number $c_n(R,a)$ of $n\times n$ circulant matrices  of  determinant  $a$    is studied.
 
 We recalled that $R$ denotes  a CFCR of nilpotency index $e$, residue field $\mathbb{F}_q$, the maximal ideal generated  $\gamma$.  By Lemma~\ref{lem:propCR}, an element  $a\in R$  can be written in the form $a =\gamma^sb$   for some $0\leq s\leq e$ and unit  $b\in U(R)$.  Precisely, $a$ is a unit if $s=0$,  $a =\gamma^sb$ is a zero-divisor if $1\leq s\leq e-1$, and $a=0$ if $s=e$.

 The over view results of the study of   $c_n(R,a)$  are summarized in Figure~\ref{D2}.   The solid  blocks are hold true for all integers $n$ such that $\gcd(n,q)=1$   and the dashed  blocks   hold true for every positive integer $n$  such that $n|(q-1)$. The details and their proofs are given right after.

 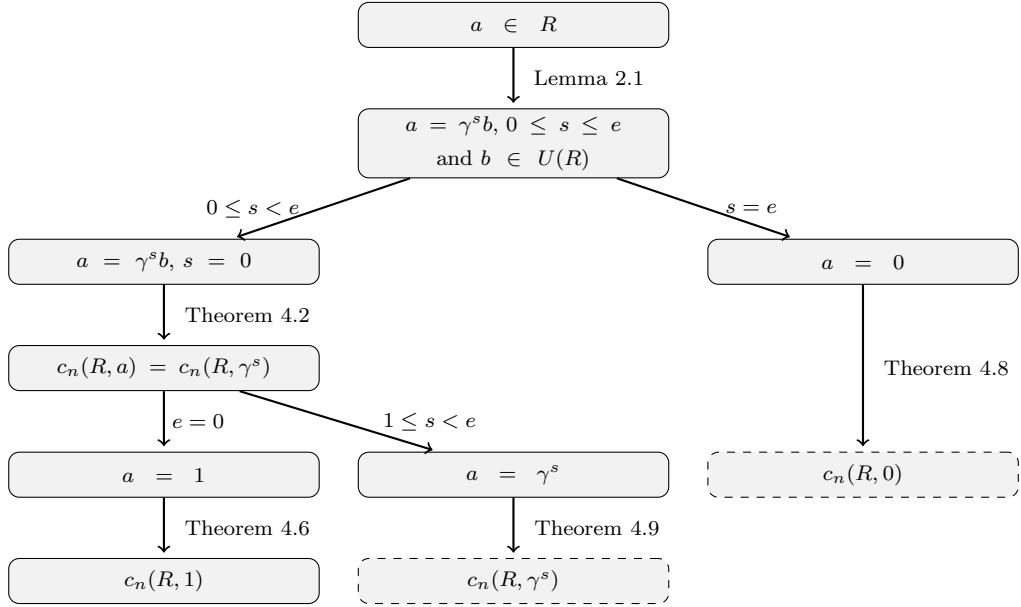
\begin{figure}[!hbt]
 	\centering {\scriptsize
 		\begin{tikzpicture} [
 		incomp/.style={rectangle, draw=black, dashed,
 			fill=black!5, text width=13em, text centered,
 			rounded corners, minimum height=2em},
 		block/.style    = { rectangle, draw=black, 
 			fill=black!5, text width=13em, text centered,
 			rounded corners, minimum height=2em },
 		line/.style     = { draw, thick, ->, shorten >=2pt },
 		]
 		\matrix [column sep=5mm, row sep=8  mm] {
 			
 			& \node [block] (oria) {$a \in R$ };            & \\ %
 			
 			& \node [block] (a) {$a=\gamma^s b$, $0\leq s \leq e$ and $b \in U(R)$ }; \\
 			
 			\node [block] (bri) {$a=\gamma^s b$, $s=0$   }; &      
 			&\node [block] (a0) {$a =0$ };  \\ 
 			
 			\node [block] (aur1) {$c_n(R,a)=c_n(R,\gamma^s)$ };  && \\
 			
 			\node [block] (au1) {$a=1$ };  	   &\node [block] (rs) {$a=\gamma^s$ };  
 			&	\node [incomp] (a01) {$c_n(R,0)$ };  \\
 			
 			\node [block] (au11) {$c_n(R,1) $ };  	   
 			&\node [incomp] (rs1) {$ c_n(R,\gamma^s)$ }; &	  \\ 	 
 		};
 		\begin{scope} [every path/.style=line]
 		\path (oria)        --   node{\quad\quad \quad\quad \quad \quad ~ Lemma~\ref{lem:propCR}}   (a); 
 		
 		\path (a)        --   node{$0\leq s <e$\quad\quad\quad \quad \quad  \quad \quad }   (bri); 
 		\path (a)        --  node{\quad \quad \quad \quad $s=e$}  (a0);

 		\path (bri)        --    node{\quad\quad\quad  \quad\quad\quad  ~ ~  Theorem~\ref{ar=r}}(aur1);  
 		
 		\path (aur1)        --    node{\quad\quad\quad $e=0$}(au1);  
 		\path (aur1)        --     node{\quad\quad\quad \quad\quad\quad \quad\quad $1 \leq s<e$}(rs);

 		\path (a0)        --    node{\quad\quad\quad \quad\quad\quad  ~ ~ Theorem~\ref{thma0}}(a01);  
 		
 		\path (au1)        --  node{\quad\quad\quad \quad\quad\quad  ~ ~  Theorem~\ref{a=1}}  (au11);  
 		\path (rs)        --   node{\quad\quad\quad \quad\quad\quad ~ ~  Theorem~\ref{a=r}} (rs1);            
 		\end{scope}
 		%
 		\end{tikzpicture}}
 	\caption{Steps in computing $c_n(R,a)$ over a CFCR  $R$} \label{D2}
 \end{figure}

 A  key relation in the study of    $c_n(R,a)$  is given in Theorem~\ref{ar=r}. First, we prove a useful lemma.

 \begin{lemma}\label{lemOnto}   Let $R$ be a CFCR  and let $n $ be a positive integer.  If   $\gcd(n,q)=1$, then  $  \{ \det(A)  \mid A\in C_n(R)\} =R$. 
 \end{lemma} 
 \begin{proof} Assume that $\gcd(n,q)=1$.   Let $Y=\{ {\rm cir}(a,b,b,\dots,b) \mid a,b\in R\}$.     Using elementary row operations, it is not difficult to see that \[\det( {\rm cir}_n(a,b,b,\dots,b)  )  =(a-b)^{n-1} (a+(n-1)b).\]
 	Let $c\in R$.    Since $\gcd(n,q)=1$, $n $ is invertible in $R$. Choose $a=(c-(1-n))n^{-1}$ and $b=(c-1)n^{-1} $ be elements in $R$. Then   $a-b=(c-(1-n))n^{-1}  -(c-1)n^{-1}=1$,  $a-(n-1)b=  (c-(1-n))n^{-1} +(n-1)(c-1)n^{-1}=c$,  and  hence,
 	\begin{align*}
 	\det( {\rm cir}_n(a,b,b,\dots,b)  )   =(a-b)^{n-1} (a+(n-1)b)   =1^{n-1}c=c.
 	\end{align*}     Since  $Y\subseteq C_n(R)$,  $R\subseteq \{ \det(A)  \mid A\in Y\} \subseteq  \{ \det(A)  \mid A\in C_n(R)\} \subseteq R$.     As desired,  we have $  \{ \det(A)  \mid A\in C_n(R)\} = R$.
 \end{proof}

 \begin{theorem} \label{ar=r} Let $R$ be a CFCR of nilpotency index $e$ and residue field $\mathbb{F}_q$.  Let $\gamma$ be a generator of the maximal ideal of $R$ and let $n$ be a positive integer.  If  $\gcd(n,q)=1$, then  \[c_n(R,b\gamma^s)=c_n(R,\gamma^s)>0\] for all units $b$ in $U(R)$ and for all integers  $0\leq s\leq e$.  
 \end{theorem}
 \begin{proof}    
 	Let  $0\leq s \leq e$ be an integer and let $b$ be a unit in $U(R)$.  In the case where $s=e$,  we have $ \gamma^s=0=\gamma^sb$. Hence,  $c_n(R,\gamma^s)= c_n(R,0)=c_n(R,b\gamma^s)$.    By Lemma~\ref{lemOnto}, we have $c_n(R,\gamma^s)>0$. 
 	
 	By Lemma~\ref{lemOnto}, there exists a matrix $B \in C_n(R)$ such that $\det(B)=b$. 	For each $0\leq s<e$, let $\alpha: C_n(R,\gamma^s)\to C_n(R,b\gamma^s)$  be a map defined by 
 	\[\alpha(A)=BA\]
 	for all $A\in C_n(R,\gamma^s)$.   Since  $B$ is nonsingular  and    $\det(A)= \gamma^s$ if  and only if $\det(  B A)=\det(  B )\det(A)  =b\gamma^s$ for all  $A\in C_n(R,\gamma^s) $,  $\alpha$ is well-defined  and  bijective.
 	Therefore, we have  \[c_n(R,\gamma^s)=|C_n(R,\gamma^s)|=|C_n(R,b\gamma^s)|=c_n(R,b\gamma^s)\] as 	 desired.
 \end{proof}
 
 By setting  $s=0$,    the following corollary can be obtained immediately. 
 \begin{corollary} \label{a1} Let $R$ be a CFCR  with residue field $\mathbb{F}_q$ and let $n$ be a positive integer such that $\gcd(n,q)=1$. Then  $c_n(R,a)=c_n(R,1)$ for all  $a\in U(R)$.
 \end{corollary}

 \subsection{Determinants of Nonsingular Circulant Matrices over CFCRs}

 Nonsingular $n\times n$  circulant matrices   over a CFCR are studied and the number      $c_n(R,a)$  is determined  for all  $a\in U(R)$ and  for all positive integers $n$ such that $\gcd(n,q)=1$.  By Corollary~\ref{a1}, it is enough  to derived  only the number $c_n(R,1)$. 
 
 Let $NSC_n(R) =\{ A\in C_n(R) \mid \det (A) \in U(R) \}$ denote the set of  nonsingular $n\times n$  circulant matrices   over a  CFCR $R$. The number  $ |NSC_n(\mathbb{F}_q)|$ which is key to determine  $c_n(R,1)$ is given in  \cite[Proposition 1]{SSPB2019}  and \cite{FGZN2018} via the ring isomorphism $T: C_n(R) \to R[X]/\langle X^n-1\rangle$  defined by \[{\rm cir} (a_1,a_2,\dots,a_n) \mapsto a_1+a_2X+a_3X^2+\dots+a_nX^{n-1}+ \langle X^n-1\rangle.\]

 \begin{lemma}[{\cite[Proposition 1]{SSPB2019}  and \cite{FGZN2018}}]\label{GLnF}  Let $q$ be a prime power and let $n$ be a positive integer  such that $\gcd(n,q)=1$. Then 
 	\[|NSC_n(\mathbb{F}_q)|
 	= \prod_{d|n} \left(q^{{\rm ord}_d(q)}-1\right)^{\frac{\phi(d)}{{\rm ord}_d(q)}} =  q^n\prod_{d|n} \left(1-q^{-{\rm ord}_d(q)}\right)^{\frac{\phi(d)}{{\rm ord}_d(q)}}. \]
 	In general,  we have 
 	\[|NSC_{np^k}(\mathbb{F}_q)| 
 	=  q^{np^k}\prod_{d|n} \left(1-q^{-{\rm ord}_d(q)}\right)^{\frac{\phi(d)}{{\rm ord}_d(q)}}\]	
 	where $k\geq 0$ is a positive integer and $p$ is the characteristic of  $\mathbb{F}_q$. 
 \end{lemma}

 The cardinality of  $NSC_n(R)$ is given in the following lemma.

 \begin{lemma} \label{GLnR} Let $R$ be a CFCR of nilpotency index $e$ and residue field $\mathbb{F}_q$ and let $n$ be a positive integer such that $\gcd(n,q)=1$.  Then \[|NSC_n(R)|  = q^{en} \prod_{d|n} \left(1-q^{-{\rm ord}_d(q)}\right)^{\frac{\phi(d)}{{\rm ord}_d(q)}}.\]
 \end{lemma}
 \begin{proof} For $e=1$,  we have $R=\mathbb{F}_q$ and   \[|NSC_n(R)|=|NSC_n(\mathbb{F}_q)|= q^n\prod_{d|n} \left(1-q^{-{\rm ord}_d(q)}\right)^{\frac{\phi(d)}{{\rm ord}_d(q)}} \]  by Lemma~\ref{GLnF}.

 	Assume that $e\geq 2$.
 	Let $\gamma$ be a generator of the maximal ideal of $R$  and let  $\beta:C_n(R)\to C_n(R/ \gamma^{e-1}R)$ be a ring homomorphism defined by 
 	\[\beta(A)=\overline{A},\]
 	where  $\overline{[a_{ij}]}:=[a_{ij} +\gamma^{e-1}R]$ for all $[a_{ij}]\in C_n(R)$.   
 	Then $A\in \ker (\beta) $ if and only if the entries of $A$ are in $\gamma^{e-1}R$. Or equivalently,  $A\in C_n(\gamma^{e-1}R)$.  By Lemma~\ref{lem:propCR},  $|\gamma^{e-1}R|=q$ which implies that $|\ker (\beta)|= |\gamma^{e-1}R|^{n}=q^{n}$.  By the $1$st Isomorphism Theorem, it follows that 
 	\begin{align*}
 	|C_n(R)|&=|\ker(\beta)||C_n(R/ \gamma^{e-1}R)|=q^{n}|C_n(R/ \gamma^{e-1}R)|.
 	\end{align*}
 	For each $B\in C_n(R/ \gamma^{e-1}R)$, we have $\beta^{-1}(B)=\{A+\ker(\beta)\}$, where  $A\in C_n(R)$ is such that $\beta(A)=B$. 
 	Note that   $A\in C_n(R)$ is invertible if and only if $\beta(A)$ is a unit  in $M_n(R/ \gamma^{e-1}R)$.  For each $B\in NSC_n(R/ \gamma^{e-1}R)$, we  therefore  have \[\beta^{-1}(B)\subseteq NSC_n(R) \text{ and }  |\beta^{-1}(B)|=|\ker(\beta)|.\]
 	It follows that   \begin{align*}
 	|NSC_n(R)|&=|\ker(\beta)||NSC_n(R/ \gamma^{e-1}R)|=q^{n}|NSC_n(R/ \gamma^{e-1}R)|.
 	\end{align*}
 	Continue this process,  it can be concluded that 
 	\begin{align*}
 	|NSC_n(R)|
 	&=q^{n}|NSC_n(R/ \gamma^{e-1}R)|\\
 	&=q^{n}q^{n}|NSC_n(R/ \gamma^{e-2}R)|\\
 	&\ \ \vdots\\
 	&=q^{(e-1)n} |NSC_n(R/ \gamma R)|\\ 
 	&=q^{(e-1)n} |NSC_n(\mathbb{F}_q)|.
 	\end{align*}
 	Since $|NSC_n(\mathbb{F}_q)|  = q^n\prod\limits _{d|n} \left(1-q^{-{\rm ord}_d(q)}\right)^{\frac{\phi(d)}{{\rm ord}_d(q)}}$	by  Lemma~\ref{GLnF},   it follows that   
 	\[|NSC_n(R)| =q^{(e-1)n} |NSC_n(\mathbb{F}_q)|= q^{en} \prod_{d|n} \left(1-q^{-{\rm ord}_d(q)}\right)^{\frac{\phi(d)}{{\rm ord}_d(q)}}\] as desired.
 \end{proof} 
 
 The number $c_n(R,1)$ is  given in the next theorem for all positive integers such that $\gcd(n,q)=1$.
 
 \begin{theorem}\label{a=1} Let $R$ be a CFCR of nilpotency index $e$ and residue field $\mathbb{F}_q$ and let $n$ be a positive integer such that $\gcd(n,q)=1$. Then  \[c_n(R,1)=q^{e(n-1)} \prod_{d|n,d\ne 1} \left(1-q^{-{\rm ord}_d(q)}\right)^{\frac{\phi(d)}{{\rm ord}_d(q)}}.\]  
 \end{theorem}
 \begin{proof}
 	Clearly,   $NSC_n(R)$ is the  disjoint union of    $C_n(R,a)$ for all $a\in U(R)$, i.e.,
 	\[NSC_n(R)=\bigcup_{a\in U(R)} C_n(R,a)\]
 	and $C_n(R,a)\cap C_n(R,b)=\emptyset$ for all $a\ne b$ in $U(R)$.
 	
 	From Corollary~\ref{a1},  we have  $c_n(R,1)=c_n(R,a)=|C_n(R,a)|$  for all units $a\in U(R)$. It follows that 
 	\begin{align*} 
 	|NSC_n(R)|&= \sum_{a\in U(R)}  |C_n(R,a)|
 	=|U(R)|c_n(R,1).
 	\end{align*}
 	By Lemma~\ref{lem:propCR} and Lemma~\ref{GLnR}, we therefore have 
 	\begin{align*} 
 	c_n(R,1)&=\frac{|NSC_n(R)|}  {|U(R)|}   \\
 	&=\frac{q^{en} \prod\limits_{d|n} \left(1-q^{-{\rm ord}_d(q)}\right)^{\frac{\phi(d)}{{\rm ord}_d(q)}}}{(q-1)q^{e-1}}\\
 	&=q^{e(n-1)} \prod_{d|n,d\ne 1} \left(1-q^{-{\rm ord}_d(q)}\right)^{\frac{\phi(d)}{{\rm ord}_d(q)}}
 	\end{align*}
 	as desired.
 \end{proof}
 
 From Corollary~\ref{a1} and Theorem~\ref{a=1}, the  next  corollary   can be deduced   directly. 
 \begin{corollary}
 	Let $R$ be a CFCR of nilpotency index $e$  and residue field $\mathbb{F}_q$ and let $n$ be a positive integer.  Then \[c_n(R,a)=c_n(R,1)=q^{e(n-1)} \prod_{d|n,d\ne 1} \left(1-q^{-{\rm ord}_d(q)}\right)^{\frac{\phi(d)}{{\rm ord}_d(q)}}\] for all units $a$ in $R$.
 \end{corollary}

 \subsection{Determinants of Singular Circulant Matrices  over CFCRs}
 
 We focus on singular $n\times n$  circulant  matrices over a CFCR $R$ with residue field $\mathbb{F}_q$ and  determine the  number  $c_n(R,a)$  of such matrices  for all $a\in \gamma R$  in the special case where $n$   is a positive divisor of  $q-1$.

 	Assume that $n$ is a divisor  of $q-1=|V\setminus \{0\}|$.  Then $R$ contains an $n$th root of unity in $V$.  Based on  \eqref{eig} and \cite[Corollary 1]{C1975}, for each $A\in C_n(R)$, $A$ is diagonalizable  and its  eigenvalues  lie in $R$. Then, for each $A\in C_n(R)$,  there exists an invertible matrix $P$ and a diagonal matrix  $D={\rm diag} (w_1,w_2, \dots,w_n)$ such that  $D=PAP^{-1}$, where $w_1,w_2,\dots, w_n\in R$  are the  eigenvalues of $A$.  Let $\theta: C_n(R) \to D_n(R)$ be the  map defined by
 	\[ A\mapsto PAP^{-1}=D\]
 	for all $A\in C_n(R)$.
 	Then $\theta$ is a determinant preserving ring isomorphism which implies that  $c_n(R,a)=d_n(R,a) $ for all $a\in R$.

 Based on the number  $d_n(R,0) $ given in  Theorem~\ref{thma0D}, the value of  $c_n(R,0) $ is given as  follows.
 \begin{theorem}  \label{thma0}  Let $R$ be a CFCR of nilpotency index $e$ and residue field $\mathbb{F}_q$  and let $n$ be a  positive  divisor of  $q-1$.
 	Then \begin{align}\label{eqa0} d_n(R,0)&=q^{ne} -(q-1)^n q^{(e-1)n} \sum_{i=0}^{e-1} \binom{n+i-1}{n-1} q^{-i}.
 	\end{align}
 \end{theorem}

 From the number $d_n(R,\gamma^s)$ in   Theorem~\ref{a=rD},  $c_n(R,\gamma^sb)= c_n(R,\gamma^s)$    is follows  for all units $b\in U(R)$ and all  integers $1\leq s<e$. 
 \begin{theorem} \label{a=r} Let $R$ be a CFCR of nilpotency index $e$ and residue field $\mathbb{F}_q$ and let  $\gamma$ be a generator of  the maximal ideal of $R$.  Then \[c_n(R,\gamma^s)=q^{(e-1)(n-1)}   (q-1)^{n-1}  \binom{n+s-1}{n-1}  \]
 	for all  integers $1\leq s<e$ and for all  positive  divisors $n$  of  $q-1$.
 \end{theorem}

 \section{Conclusion and Remarks}
 \label{sec:concl}
 
 Summary, conjectures, open problems are given in  this section together with    a brief discussion on the determinants of diagonal and circulant matrices over  commutative finite principal ideal rings  (CFPIRs).

 \subsection{Conclusion}
 Determinants of  $n\times n$ matrices over  CFCRs $R$ with residue field $\mathbb{F}_q$  have been established  in \cite{CJU2017}.  Here,  determinants of $n\times n$ diagonal and circulant matrices over $R$ which are subrings of  the matrices in \cite{CJU2017} have been studied.  The number of  $n\times n$ diagonal matrices over $R$  of a fixed determinant  $a$ have been completely determined for all positive integers $n$ and for all elements $a\in R$.  For  nonsingular  circulant matrices,  the number of  $n\times n$ circulant  matrices over $R$ whose determinant is a unit in $R$  have been completely determined for all units in $R$ and for all positive integers $n$ such that $\gcd(n,q)=1$.   For singular   circulant matrices,  the number of  $n\times n$ circulant matrices over $R$ whose determinant is a non-unit  have been completely determined for all non-units in $R$   and positive  divisors $n$  of  $q-1$. The other  cases remain  as  open problems.

 \subsection{Determinants of Diagonal and Circulant Matrices over CFPIRs}
 
 A ring $\mathcal{R}$ with identity $1\ne  0$ is called a {\em commutative finite principal ideal ring} (CFPIR) if  $\mathcal{R}$  is finite, commutative,  and every ideal in  $\mathcal{R}$ is principal.  It is  well known (see\cite{GF2002}) that 
 every CFPIR is a direct product of CFCRs.
 Hence, a  CFPIR $\mathcal{R}$ can be written in the form of  $\mathcal{R}=R_1\times R_2\times \dots\times R_m$ for some positive  integer $m$, where $R_i$ is a CFCR for all $1\leq i\leq m$.  
 
 For each $1\leq i\leq m$, let $\phi_i:\mathcal{R}\to R_i$ be a  surjective ring homomorphism defined by \[\phi_i((r_1,r_2,\dots,r_m))=r_i.\]  
 Let $M_n( \mathcal{R})$ be the ring of $n\times n$ matrices over $\mathcal{R}$. 
 In \cite[Theorem 4.1]{CJU2017}, it has been shown that  the map $\Phi:  M_n( \mathcal{R})\to   M_n( {R}_1) \times M_n( {R}_2)\times \dots \times M_n( {R}_m)   $ defined by 
 \[  [a_{ij}] \mapsto ([\phi_1(a_{ij})], [\phi_2(a_{ij})] ,\dots, [\phi_m(a_{ij})])\]
 is a ring isomorphism.  Since $D_n(\mathcal{R} )$ and  $C_n(\mathcal{R} )$ are   subrings of $M_n( \mathcal{R})$,  
 the restriction maps   $\Phi{\vert }_{D_n( \mathcal{R})}:  D_n( \mathcal{R})\to   D_n( {R}_1) \times D_n( {R}_2)\times \dots \times D_n( {R}_m)   $ 
 and  $\Phi{\vert }_{C_n( \mathcal{R})}:  C_n( \mathcal{R})\to   C_n( {R}_1) \times C_n( {R}_2)\times \dots \times C_n( {R}_m)   $  are ring isomorphisms. 
 Consequently, the following theorem can be obtained using the arguments similar to those in the proof of \cite[Theorem 4.1]{CJU2017}.
 \begin{theorem} Let  $\mathcal{R}=R_1\times R_2\times \dots \times R_m$ be a CFPIR where $R_1, R_2, \dots, R_m$  be  CFCRs and let $n$ be a positive integer. Let $r\in \mathcal{R}$ and let  $\phi_i$'s  be defined as above.  Then \[ d_n(\mathcal{R},r)=d_n(R_1,\phi_1(r))d_n(R_2,\phi_2(r))\dots d_m(R_m,\phi_m(r) )\]
 	and
 	\[ c_n(\mathcal{R},r)=c_n(R_1,\phi_1(r))c_n(R_2,\phi_2(r))\dots c_m(R_m,\phi_m(r) ).\]
 \end{theorem}

 \subsection{Conjectures and Open Problems}
 Based on our observation,   we conjecture that the condition $\gcd(n,q)=1$ in Theorem~\ref{ar=r} can be omitted.  However, the original proof of  Theorem~\ref{ar=r} does not work. 
 \begin{conjecture}   Let $R$ be a CFCR of nilpotency index $e$   and let $n$ be a positive integer.     If the maximal ideal of $R$ is generated by $\gamma$, then  \[c_n(R,\gamma^s)=c_n(R,b\gamma^s) \] for all units $b$ in $U(R)$ and   $0\leq s\leq e$.   This number can be zero.
 \end{conjecture}

 In Lemma~\ref{lemOnto},  $ \{ \det(A)  \mid A\in C_n(R)\} =R$  for all  positive integers $n$ such that  $\gcd(n,q)=1$.  Once  $\gcd(n,q)\ne 1$,    $ \{ \det(A)  \mid A\in C_n(R)\} $ does not need to equal $R$,  e.g.,  there are  no $2\times 2 $ matrices over $\mathbb{Z}_4$ whose determinant is a zero-divisor $2$.    We have the following conjecture. 
 
 \begin{conjecture}   Let $R$ be a CFCR of nilpotency index $e$   and let $n$ be a positive integer.  Then $U(R)  \subseteq \{ \det(A)  \mid A\in C_n(R)\} $. 
 	Precisely,   $c_n(R,a) >0$ for all units $a\in U(R)$.
 \end{conjecture} 
 
 In general, it is interesting to solve the following open problems.
 \begin{problem} For a CFCR $R$ with residue field $\mathbb{F}_q$, determine $c_n(R,a) $ for all   $a\in U(R)$ and for all positive  integers $n$ such  $\gcd(n,q)\ne 1$. 
 \end{problem}
 
 \begin{problem} For a CFCR $R$ with residue field $\mathbb{F}_q$, determine $c_n(R,a) $ for all  non-units $a\in \gamma R$ and for all positve integers $n$   which are not a divisor of  $q-1$.
 \end{problem}

 \begin{problem} For a commutative  finite ring $\mathfrak{R}$
 	and $a\in  \mathfrak{R}$,  determine $d_n(\mathfrak{R},a)$ and  $c_n(\mathfrak{R},a)$   for all positive integers $n$.
 \end{problem}

\end{document}